\documentclass[english,leqno]{amsart}
\usepackage{amsmath,amssymb}
\usepackage{amsthm}
\usepackage{proof}

\numberwithin{equation}{section}

\newtheorem{Th}{Theorem}[section]
\newtheorem{Lem}{Lemma}[section]
\newtheorem{Cor}{Corolary}[section]
\theoremstyle{definition}
\newtheorem{Con}{Condition}[section]
\newtheorem*{Ack}{Acknowledgement}

\newcommand{\op}{\oplus}
\newcommand{\AC}{\textup{\textrm{AC}}}
\newcommand{\co}{\mathop{\textup{\textrm{co}}}}
\newcommand{\clco}{\mathop{\textup{\textrm{clco}}}}
\newcommand{\Equil}{\mathop{\textup{\textrm{Equil}}}}

\newcommand{\varmid}{\mathrel{}\middle|\mathrel{}}

\title[Resolvents of equilibrium problems in a geodesic space]{Resolvents of equilibrium problems in a complete geodesic space with negative curvature}
\author[Y.~Kimura]{Yasunori Kimura}\address[Yasunori Kimura]{Department of Information Science, Toho University, Miyama, Funa\-bashi, Chiba 274--8510, Japan}\email{yasunori@is.sci.toho-u.ac.jp}
\author[T.~Ogihara]{Tomoya Ogihara}\address[Tomoya Ogihara]{Department of Information Science, Toho University, Miyama, Funa\-bashi, Chiba 274--8510, Japan}\email{6522005o@st.toho-u.jp}
\date{}

\begin{document}
\maketitle
\begin{abstract}
In this paper, we propose a resolvent of an equilibrium problem in a geodesic space with negative curvature having the convex hull finite property. We prove its well-definedness as a single-valued mapping whose domain is whole space, and study the fundamental properties.
\end{abstract}
\section{Introduction}
Let $K$ be a nonempty set and $f\colon K \times K \to \mathbb{R}$. An equilibrium problem is defined as to find $z_0 \in K$ such that
\[
f(z_0,y) \geq 0
\]
for all $y \in K$. Equilibrium problems were first studied by Blum and Oettli in Banach spaces; see \cite{BlumOettli}. In 2005, Combettes and Hirstoaga proposed the resolvent of equilibrium problems in Hilbert spaces.
\begin{Th}[Combettes and Hirstoaga \cite{Com}]
Let $H$ be a Hilbert space and $K$ a nonempty closed convex subset of $H$. Let $f\colon K \times K \to \mathbb{R}$ and $S_f$ the set of solutions to the equilibrium problem for $f$. Suppose the following conditions:
\begin{itemize}
\item $f(y,y) = 0$ for all $y \in K$;
\item $f(y,z) + f(z,y) \leq 0$ for all $y,z \in K$;
\item $f(y,\cdot)\colon K \to \mathbb{R}$ is lower semicontinuous and convex for every $y \in H$;
\item $f(\cdot,z)\colon K \to \mathbb{R}$ is upper hemicontinuous for every $y \in K$.
\end{itemize}
Then the resolvent operator $J_f$ defined by
\[
J_fx = \left\{z \in K\varmid \inf_{y \in K}(f(z,y) + \left<z-x,y-z\right>)\geq 0\right\}
\]
has the following properties:
\begin{enumerate}
\item[\textup{(i)}] $D(J_f) = X$;
\item[\textup{(ii)}] $J_f$ is single-valued and firmly nonexpansive;
\item[\textup{(iii)}] $F(J_f) = S_f$;
\item[\textup{(iv)}] $S_f$ is closed and convex.
\end{enumerate}
\end{Th}
Recently, the resolvents of equilibrium problems were proposed in geodesic spaces. In a CAT$(0)$ space with the convex hull finite property, it was proposed by Kimura and Kishi in 2018; see \cite{KimuraKishi}. In  an admissible CAT$(1)$ space with the convex hull finite property, it was proposed by Kimura in 2021; see \cite{Kimura2021}.

In this paper, we propose a different type of the resolvent of the equilibrium problems and a $\Delta$-convergence theorem using the proximal point algorithm in a CAT$(-1)$ space with the convex hull finite property. In Section 2, we see the definition of  a CAT$(\kappa)$ space for $\kappa \in \mathbb{R}$, fundamental properties of a CAT$(0)$ space and a CAT$(-1)$ space, and  some important properties to prove the main results. In Section 3, we prove that the resolvent of equilibrium problems is well-defined and show some properties of this resolvent in a CAT$(-1)$ space with the convex hull finite property.

\section{Preliminaries}
Let $(X,d)$ be a metric space, $T$ a mapping of $X$ into itself. We denote by $F(T)$ the set of all fixed points of $T$. For $x,y \in X$, a mapping $c$ of $[0,d(x,y)]$ into $X$ is called \emph{a geodesic with endpoints $x,y \in X$} if it satisfies the following:
\begin{itemize}
\item $c(0) = x$, $c(d(x,y)) = y$;
\item $d(c(s),c(t)) = |s-t|$ for all $s,t \in [0,d(x,y)]$.
\end{itemize}

For $ D \in \mathopen]0,\infty]$, we call $X$ \emph{a $D$-geodesic space} if a geodesic with endpoints $x$ and $y$ exists for all $x,y \in X$ with $d(x,y) < D$. In what follows, when $X$ is a $D$-geodesic space, every geodesic in $X$ whose length is less than $D$ is always supposed to be unique. We call such $X$ \emph{a uniquely $D$-geodesic space}. Then, the image of a geodesic with endpoints $x$ and $y$ is denoted by $[x,y]$ for all $x,y \in X$ with $d(x,y)<D$. For $x,y \in X$ with $d(x,y) < D$ and $t \in [0,1]$, there exists a unique point $z \in [x,y]$ such that $d(x,z) = (1-t)d(x,y)$ and $d(y,z) = td(x,y)$. We denote it by $z = tx \op (1-t)y$.
For $\kappa \in \mathbb{R}$, a $2$-dimensional model space $M_{\kappa}^2$ with a curvature $\kappa$ is  defined by 
\[
M_{\kappa}^2 =
\begin{cases}
\dfrac{1}{\sqrt{\kappa}}\mathbb{S}^2 & (\kappa > 0);\\[10pt]
\mathbb{E}^2 & (\kappa = 0);\\[5pt]
\dfrac{1}{\sqrt{-\kappa}}\mathbb{H}^2 & (\kappa < 0),
\end{cases}
\]
where $\mathbb{E}^2$ is the two-dimensional Euclidean space, $\mathbb{H}^2$ is the two-dimensional hyperbolic space and $\mathbb{S}^2$ is the two-dimensional unit sphere. The diameter of $M_{\kappa}^2$ is denoted by $D_{\kappa}$, that is, 
\[
D_{\kappa}=
\begin{cases}
\dfrac{\pi}{\sqrt{\kappa}} & (\kappa > 0);\\[5pt]
\infty & (\kappa \leq 0).
\end{cases}
\]
For $x,y,z \in X$, \emph{a geodesic triangle $\triangle(x,y,z)$} with vertices $x,y,z \in X$ is defined by $[x,y] \cup [y,z] \cup [z,x]$. For $x,y,z \in X$ with $d(x,y) + d(y,z) + d(z,x) < 2D_{\kappa}$, \emph{a comparison triangle} to $\triangle(x,y,z) \subset X$ of vertices $\bar{x},\bar{y},\bar{z} \in M_{\kappa}^2$ is defined by $[\bar{x},\bar{y}]\cup[\bar{y},\bar{z}]\cup[\bar{z},\bar{x}]$ with $d(x,y) = d_{M_{\kappa}^2}(\bar{x},\bar{y})$, $d(y,z) = d_{M_{\kappa}^2}(\bar{y},\bar{z})$ and $d(z,x)=d_{M_{\kappa}^2}(\bar{z},\bar{x})$. It is denoted by $\overline\triangle(\bar{x},\bar{y},\bar{z}) \subset M_{\kappa}^2$. A point $\bar{p} \in [\bar{x},\bar{y}]$ is called \emph{a comparison point} for $p \in [x,y]$ if $d(x,p) = d_{{M_{\kappa}^2}}(\bar{x},\bar{p})$. If for all $p,q \in \triangle (x,y,z)$ and their comparison points $\bar{p},\bar{q} \in \overline{\triangle}(\bar{x},\bar{y},\bar{z})$, the inequality $d(p,q) \leq d_{M_{\kappa}^2}(\bar{p},\bar{q})$ holds for all triangles in $X$, then we call $X$ a CAT$(\kappa)$ space. In general, for real numbers $\kappa,\kappa'$ with $\kappa' < \kappa$, a CAT$(\kappa')$ space is a CAT$(\kappa)$ space; see \cite{Bridson}. In a CAT$(-1)$ space, the following inequalities holds:
\begin{Lem}[Bridson and Haeflinger \cite{Bridson}]\label{PL}
Let $X$ be a complete \textup{CAT$(-1)$} space, $x,y,z \in X$ and $t \in [0,1]$. Then, the following inequality holds:
\begin{align*}
&\cosh d(tx \op (1-t)y,z)\sinh d(x,y)\\
&\leq \cosh d(x,z)\sinh td(x,y) + \cosh d(y,z)\sinh(1-t)d(x,y).
\end{align*}
\end{Lem}

\begin{Cor}[Bridson and Haeflinger \cite{Bridson}]\label{coshconvex}
Let $X$ be a complete \textup{CAT$(-1)$} space, $x,y,z \in X$ and $t \in [0,1]$. Then, the following inequality holds:
\[
\cosh d(tx \op (1-t)y,z) \leq t\cosh d(x,z) + (1-t)\cosh d(y,z).
\]
\end{Cor}

Let $X$ be a CAT$(-1)$ space and $T$ a mapping of $X$ into itself. Then $T$ is \emph{firmly hyperbolically nonspreading} if the inequality
\[
(\cosh d(x,Tx) + \cos d(y,Ty))\cosh d(Tx,Ty) \leq \cosh d(Tx,y) + \cosh d(x,Ty)
\]
holds for $x,y \in X$. Further, $T$ is \emph{hyperbolically nonspreading} \cite{KajimuraKimura2019} if the inequality
\[
2\cosh d(Tx,Ty) \leq \cosh d(x,Ty) + \cosh d(Tx,y)
\]
holds for all $x,y \in X$. A mapping $T$ is \emph{quasinonexpansive} if $F(T)$ is nonempty and the inequality $d(Tx,z) \leq d(x,z)$ holds for $x\in X$ and $z \in F(T)$. Then, firmly hyperbolically nonspreading mappings are hyperbolically nonspreading,  and hyperbolically nonspreading mappings with nonempty $F(T)$ are quasinonexpansive.

Let $X$ be a CAT(0) space, $\{x_{\alpha}\}$ a bounded net in $X$, and 
\[
r(x,\{x_\alpha\}) = \limsup_{\alpha}d(x,x_\alpha) 
\]
for $x \in X$. Then \emph{the asymptotic radius $r(\{x_\alpha\})$ of $\{x_\alpha\}$} is defined by
\[
r(\{x_\alpha\}) = \inf_{x \in X}r(x,\{x_\alpha\}).
\]
A point $x_0 \in X$ is an \emph{asymptotic center of $\{x_\alpha\}$} if $r(\{x_\alpha\}) = r(x_0,\{x_\alpha\})$. The set of all asymptotic centers is denoted by $\AC(\{x_\alpha\})$. We say that $\{x_\alpha\}$ is \emph{$\Delta$-convergent} to $x_0\in X$ if $x_0$ is a unique asymptotic center of all subnets of $\{x_\alpha\}$. We know that every bounded net $\{x_\alpha\}$ in a complete CAT$(0)$ space has a unique asymptotic center; see  \cite{Kirk2006}
\begin{Lem}[Kirk and Panyanak \cite{Kirk2008}]
Let $X$ be a complete \textup{CAT(0)} space. Then every bounded net has a subnet which is $\Delta$-convergent to $x_0 \in X$.
\end{Lem}

Let $X$ be a geodesic space and $f$ a function of $X$ into $\mathbb{R}$. A function $f$ is said to be \emph{lower semicontinuous} if the set $\left\{x \in X\mid f(x) \leq a\right\}$ is closed for all $a \in \mathbb{R}$. If $f$ is continuous, then it is lower semicontinuous. A function $f$ is said to be \emph{convex} if
\[
f(\alpha x \op (1-\alpha)y) \leq \alpha f(x) + (1-\alpha)f(y)
\]
holds for all $x,y \in X$ and $\alpha \in \mathopen]0,1\mathclose[$. A function $f$ is said to be \emph{upper hemicontinuous} if the inequality
\[
f(x) \geq \limsup_{t \to 0^+}f((1-t)x \op ty)
\]
 holds for all $x,y \in X$ and $t \in \mathopen]0,1\mathclose[$.
 \begin{Lem}[Mayer \cite{Mayer}]\label{Mayer}
Let $X$ be a complete \textup{CAT(0)} space and $f$ a lower semicontinuous convex function from $X$ into $\mathbb{R}$. Then there exists a nonnegative real number $c$ which depends on only $f$ such that
\[
\liminf_{d(u,v) \to \infty}\frac{f(u)}{d(u,v)} \geq -c
\]
for all $v \in X$.
\end{Lem}

Let $X$ be a complete CAT$(0)$ space and $C$ a subset of $X$. Then $C$ is said to be \emph{$\Delta$-compact} if every net $\{x_{\alpha}\}$ of $C$ has a subnet $\Delta$-converging to a point in $C$. It is known that every bounded closed convex set in CAT$(0)$ space is $\Delta$-compact; see \cite{Kirk2008}. Let $E$ be a set of $X$. Then \emph{a convex hull of $E$} is defined by
\[
\co E = \bigcup_{n=0}^{\infty}X_n,
\]
where $X_0 = E$ and $X_n = \left\{tu_{n-1} \op (1-t)v_{n-1} \varmid u_{n-1},v_{n-1} \in X_{n-1}, t \in[0,1]\right\}$.

A complete CAT$(0)$ space has \emph{the convex hull finite property} if every continuous mapping $T$ of $\clco E$ into itself has a fixed point for all finite subsets $E$ of $X$, where $\clco E$ is the closure of $\co E$; see \cite{Shabanian2011}.

The following lemmas show that the KKM lemma holds on a complete CAT$(0)$ space with the convex hull finite property.
\begin{Lem}[Niculescu and Roven\c{t}a \cite{Niculescu}]\label{KKM}
Let $X$ be a complete \textup{CAT$(0)$} space having the convex hull finite property and $C$ a nonempty subset of $X$. Suppose that for every $x \in C$, there exists a closed subset $M(x)$ of $C$ such that $\co E\subset\bigcup_{x \in E}M(x)$ for all finite subsets $E$ of $C$. Then $\bigcap_{i=1}^nM(y_{i})$ is nonempty for every finite subset $\{y_1,y_2,\dots,y_n\}$ of $X$.
\end{Lem}
\begin{Lem}[Kimura and Kishi \cite{KimuraKishi}]\label{KK}
Let $X$ be a complete \textup{CAT$(0)$} space having the convex hull finite property and $C$ a $\Delta$-compact subset of $X$. Then if $\{F_i\}_{i \in I}$ is a family of $\Delta$-closed subsets of $C$ with the finite intersection property, then $\bigcap_{i \in I}F_i$ is nonempty.
\end{Lem}

\begin{Con}\label{bifunction}
Let $X$ be a complete \textup{CAT$(-1)$} space and $K$ a nonempty closed convex subset of $X$. We suppose that a bifunction $f\colon K \times K \to \mathbb{R}$ satisfies the following conditions:
\begin{enumerate}
\item[(i)] $f(x,x) = 0$ for all $x \in K$;
\item[(ii)] $f(x,y) + f(y,x) \leq 0$ for all $x,y \in K$;
\item[(iii)] for every $x \in K$, $f(x,\cdot)\colon K \to \mathbb{R}$ is lower semicontinuous and convex;
\item[(iv)]for every $y \in K$, $f(\cdot,y)\colon K \to \mathbb{R}$ is upper hemicontinuous.
\end{enumerate}
\end{Con}

The set of solutions to the equilibrium problem for $f$ is denoted by $\Equil f$, that is, 
\[
\Equil f = \left\{z \in K\varmid \inf_{y \in K}f(z,y) \geq 0\right\}.
\]
\section{Main Result}
In this section, we show that the resolvent of equilibrium problems is well-defined in a complete CAT$(-1)$ space.
\begin{Lem}\label{compact}
Let $X$ be a complete \textup{CAT$(-1)$} space having the convex hull finite property, $K$ a nonempty closed convex subset of $X$, and $C$ a nonempty $\Delta$-compact closed convex subset of $K$. Suppose that $f\colon K \times K \to \mathbb{R}$ satisfies \textup{Condition \ref{bifunction}}. Then, there exists $z \in C$ such that 
\[
f(y,z) \leq \cosh d(x,y) - \cosh d(x,z)
\]
for all $x \in X$ and $y \in C$.
\end{Lem}
\begin{proof}
For arbitrarily fixed $x \in X$, let 
\[
h(y,u) = f(y,u)+\cosh d(x,u)-\cosh d(x,y)
\]
for $u,y \in K$.
Put
\[
M(y) = \left\{u \in C \varmid h(y,u) \leq 0\right\}
\]
for all $y \in C$. Then $M(y)$ is $\Delta$-closed. We show that $\bigcap_{y \in C}M(y)$ is nonempty. Let $\{y_i\}_{i \in N}$ be a subset of $C$ with a finite index set $N$ and $I$ a nonempty subset of $N$. Put $D = \{y_{i}\}_{i \in I}$. Let $p \in \co D$ arbitrarily, where $\co D$ is defined by
\[
\co D = \bigcup_{n = 0}^{\infty}F_n
\]
wherever $F_0 = D$ and $F_n = \left\{tu_{n-1} \op (1-t)v_{n-1} \varmid u_{n-1},v_{n-1} \in F_{n-1}, t \in [0,1]\right\}$ for $n \in \mathbb{N}$. We first show the following inequarities by induction. For $n \in \mathbb{N}\cup\{0\}$, if $p \in F_n$, then there exists $\left\{\mu_j\in\mathbb{R}\varmid j \in I\right\}$ such that
\begin{align}
h(y_i,p) \leq \sum_{j \in I}\mu_jh(y_i,y_j)\label{1}
\end{align}
where $\sum_{j \in I}\mu_j = 1$ and $\mu_j \geq 0$. If $n = 0$, then for $p \in F_0$, we get $p = y_{i_0}$ for some $i_0 \in I$. Put
\[
\mu_{j}=
\begin{cases}
1 & (j = i_0);\\
0 & (j \neq i_0).
\end{cases}
\]
Then, we get 
\[
h(y_i,p)=h(y_i,y_{i_0})=\sum_{j\in I}\mu_jh(y_i,y_j).
\]
Suppose \eqref{1} holds for fixed $n \in \mathbb{N}$. Suppose $p \in F_{n+1}$. Then there exists $u,v \in F_n$ and $t \in [0,1]$ such that $p = tu \op (1-t)v$ and we get
\begin{align*}
h(y_i,p) \leq th(y_i,u) + (1-t)h(y_i,v).
\end{align*}
Hence there exist $\lambda_i, \nu_i \in [0,1]$ such that
\begin{align*}
th(y_i,u) + (1-t)h(y_i,v) 
&\leq t\sum_{j\in I}\lambda_jh(y_i,y_j) + (1-t)\sum_{j\in I}\nu_jh(y_i,y_j)\\
&=\sum_{j\in I}(t\lambda_j+(1-t)\nu_j)h(y_i,y_j)
\end{align*}
and $\sum_{j \in I} \lambda_j = \sum_{j \in I}\nu_j = 1$. Put $\mu_j = t\lambda_j + (1-t)\nu_j$. Then we have \eqref{1} holds for $n+1$. Therefore it holds for every $n \in \mathbb{N}$. Suppose that $h(y_i,p)>0$ for all $i \in I$. Then we get
\[
0 <\sum_{i\in I}\mu_ih(y_i,p) \leq \sum_{i\in I}\sum_{j\in I}\mu_i\mu_jh(y_i,y_i) =\frac{1}{2}\sum_{i\in I}\sum_{j\in I}\mu_i\mu_j(h(y_i,y_j) + h(y_j,y_i)) \leq 0
\]
and this is a contradiction. We have there exists $i \in I$ such that $h(y_i,p)\leq0$ and $p \in M(y_i)$. Since $p \in \co D$ is arbitrary, we obtain $\co D \subset \bigcup_{i \in I}M(y_i)$. By Lemma \ref{KKM}, we get $\bigcap_{i \in N}M(y_i)$ is nonempty. Using  Lemma \ref{KK}, we obtain $\bigcap_{y \in C}M(y)$ is nonempty. Therefore we can take $ z \in \bigcap_{y \in C}M(y)$, which satisfies the desired result.
\end{proof}

\begin{Lem}\label{bifunctionslemma}
Let $X$ be a complete \textup{CAT$(-1)$} space having the convex hull finite property, $K$ a nonempty closed convex subset of $X$, and $C$ a nonempty $\Delta$-compact and closed convex subset of $K$. Suppose that $f\colon K \times K \to \mathbb{R}$ satisfies \textup{Condition \ref{bifunction}}. Then for $x \in X$ and $z \in C$, the following $(1)$ and $(2)$ are equivalent:
\begin{enumerate}
\item[\textup{$(1)$}] $f(y,z) \leq \cosh d(x,y) - \cosh d(x,z)$ for all $y \in C$;
\item[\textup{$(2)$}] $0 \leq f(z,y) + \cosh d(x,y) - \cosh d(x,z)$ for all $y\in C$.
\end{enumerate}
\end{Lem}
\begin{proof}
If the statement $(2)$ holds, by Condition \ref{bifunction}, we get $f(z,y) \leq -f(y,z)$ and hence we get the statement $(1)$. Inversely, suppose the statement $(1)$ holds. Put $\tau_t = ty \op (1-t)z$ for $t \in \mathopen]0,1\mathclose[$. By Condition \ref{bifunction} and Lemma \ref{compact}, we get
\begin{align*}
0 
&= f(\tau_t,\tau_t)\\
&\leq tf(\tau_t,y) + (1-t)f(\tau_t,z)\\
&\leq tf(\tau_t,y) + (1-t)(\cosh d(x,\tau_t) - \cosh d(x,z))\\
&\leq tf(\tau_t,y)  + (1-t)(t\cosh d(x,y) + (1-t)\cosh d(x,z) - \cosh d(x,z))\\
&\leq tf(\tau_t,y) + (1-t)t(\cosh d(x,y) - \cosh d(x,z)).
\end{align*}
Dividing both sides by t, letting $t \to 0^+$ and using the upper hemicontinuity of $f$, we get 
\[
0 \leq f(z,y) + \cosh d(x,y) - \cosh d(x,z)
\]
and hence we get the statement $(2)$.
\end{proof}

\begin{Lem}\label{exists}
Let $X$ be a complete \textup{CAT$(-1)$} space having the convex hull finite property, $K$ a nonempty closed convex subset of $X$, and $C$ a nonempty $\Delta$-compact closed convex subset of $K$. Suppose that $f\colon K \times K \to \mathbb{R}$ satisfies \textup{Condition \ref{bifunction}}. Then there exists $z \in C$ such that
\[
0 \leq f(z,y) + \cosh d(x,y) - \cosh d(x,z)
\]
for all $x \in X$ and $y \in C$.
\end{Lem}

\begin{proof}
By Lemma \ref{compact}, there exists $z \in C$ such that
\[
f(y,z) \leq \cosh d(x,y) - \cosh d(x,z)
\]
for all $y \in C$. Further, by Lemma \ref{bifunctionslemma}, we get
\[
0 \leq f(z,y) + \cosh d(x,y) - \cosh d(x,z) 
\]
for all $y \in C$. Consequently, we obtain the desired result.
\end{proof}
We can prove that the resolvent of an equilibrium problem is well-defined.
\begin{Th}\label{WelldefinedResolvent}
Let $X$ be a complete \textup{CAT$(-1)$} space having the convex hull finite property and $K$ a nonempty closed convex subset of $X$.  Suppose that $f\colon K \times K \to \mathbb{R}$ satisfies \textup{Condition \ref{bifunction}}. Define a set-valued mapping $L_f \colon X \to 2^K$ by 
\[
L_fx = \left\{z \in K\varmid \inf_{y\in K}(f(z,y) + \cosh d(x,y) - \cosh d(x,z)) \geq 0\right\}
\]
for all $x \in X$. Then the following hold:
\begin{enumerate}
\item[\textup{(i)}] $D(L_f) = X$;
\item[\textup{(ii)}] $L_f$ is single-valued and firmly hyperbolically nonspreading;
\item[\textup{(iii)}]$\Equil f = F(L_f)$, and thus $\Equil f$ is closed and convex.
\end{enumerate}
\end{Th}
\begin{proof}
(i) Fix $x \in K$. We show that there exists $z \in K$ such that
\[
0 \leq f(z,y) + \cosh d(x,y) - \cosh d(x,z)
\]
for all $y \in K$. Fix $a \in K$. Then for $w \in K$, we get
\begin{align*}
 &f(w,a) + \cosh d(x,a) - \cosh d(x,w)\\
&\leq -f(a,w) + \cosh d(x,a) - \cosh d(x,w)\\
&=\cosh d(x,a) -(f(a,w)+\cosh d(x,w)).
\end{align*}
By Lemma \ref{Mayer}, there exists $c \geq 0$ such that
\[
\liminf_{d(a,w)\to\infty}\frac{f(a,w)}{d(a,w)} \geq -c.
\]
Then, we get 
\[
\liminf_{d(a,w)\to\infty}\frac{f(a,w)+\cosh d(x,w)}{d(a,w)}\geq -c + \liminf_{d(a,w)\to\infty}\frac{\cosh d(x,w)}{d(a,w)}=\infty
\]
and hence $f(a,w)+\cosh d(x,z) \to\infty$ as $d(a,w)\to\infty$. This implies that
\[
f(w,a) + \cosh d(x,a) - \cosh d(x,w) \to-\infty
\]
as $d(a,w)\to\infty$. Then, we can take large $R > 0$ such that
\[
f(w,a) + \cosh d(x,a) - \cosh d(x,w) \leq 0
\]
wherever $d(w,a) = R$ for $w \in K$. Let $C = \left\{w \in K\varmid d(w,a) \leq R\right\}$. Then $C$ is bounded, closed and convex. Hence $C$ is $\Delta$-compact. By Lemma \ref{exists}, there exists $z_0 \in C$ such that
\[
0 \leq f(z_0,y) + \cosh d(x,y) - \cosh d(x,z_0)
\]
for all $y \in C$. We next show
\[
0\leq f(z_0,y) + \cosh d(x,y)-\cosh d(x,z_0)
\]
for each $y \in K\setminus C$. Let $y\in K\setminus C$. Then, we have $d(a,y)>R$. Let 
\[
u_0=
\begin{cases}
a & (d(a,z_0)=R),\\
z_0 & (d(a,z_0) < R).
\end{cases}
\]
Then, we get $d(a,u_0) < R$. In fact, if $d(a,z_0)=R$, then we have $d(a,u_0)=d(a,a)=0$. On the other hand, if $d(a,z_0)<R$, then we have $d(a,u_0) = d(a,z_0) < R$. Since $d(a,y) > R$ and $d(a,u_0) < R$, we can take sufficiently small $t_0 \in \mathopen]0,1\mathclose[$ satisfying $t_0d(a,y)+(1-t_0)d(a,u_0) < R$. Then, we get
\[
d(a,t_0y\op(1-t_0)u_0) \leq t_0d(a,y) + (1-t_0)d(a,u_0) < R
\]
and hence $d(a,t_0y\op(1-t_0)u_0)<R$. Since $K$ is convex, we get $t_0y\op(1-t_0)u_0\in K$. This implies that $t_0y\op(1-t_0)u_0 \in C$. Therefore
\begin{align*}
0
&\leq f(z_0,t_0y\op(1-t_0)u_0)+ \cosh d(x,t_0y\op(1-t_0)u_0) - \cosh d(x,z_0)\\
&\leq t_0(f(z_0,y) + \cosh d(x,y)-\cosh d(x,z_0))\\
&\hspace{10pt} + (1-t_0)(f(z_0,u_0) + \cosh d(x,u_0)-\cosh d(x,z_0)).
\end{align*}
Further, we get 
\[
f(z_0,u_0) + \cosh d(x,u_0)-\cosh d(x,z_0) \leq 0.
\]
Indeed, if $d(a,z_0)=R$, we get 
\begin{align*}
f(z_0,u_0) + \cosh d(x,u_0) - \cosh d(x,z_0)
&=f(z_0,a) + \cosh d(x,a)-\cosh d(x,z_0) \leq 0.
\end{align*}
On the other hand, if $d(a,z_0)<R$, then $u_0=z_0$ and thus we get
\[
f(z_0,u_0) + \cosh d(x,u_0) - \cosh d(x,z_0) = 0.
\]
Then, we get
\begin{align*}
&f(z_0,y) + \cosh d(x,y) - \cosh d(x,z_0)\\
&\geq -\frac{1-t_0}{t_0}(\cosh d(z_0,u_0) + \cosh d(x,u_0)-\cosh d(x,z_0))\geq 0.
\end{align*}
Since $y \in K\setminus C$ is arbitrary, we get 
\[
f(z_0,y) + \cosh d(x,y) - \cosh d(x,z) \geq 0
\]
for all $y\in K$.

(ii) Let $x \in X$, $w \in K$ and $z \in L_fx$ with $w \neq z$. Putting $\tau_t = tw \op (1-t)z$ for $t \in \mathopen]0,1\mathclose[$, by Lemma \ref{PL}, we get
\begin{align*}
0
&\leq f(z,\tau_t) + \cosh d(x,\tau_t) - \cosh d(x,z)\\
&\leq tf(z,w) + \cosh d(x,\tau_t) - \cosh d(x,z)\\
&= tf(z,w) + \frac{1}{\sinh d(w,z)}(\cosh d(x,\tau_t)\sinh d(w,z) - \cosh d(x,z)\sinh d(w,z))\\
&\leq tf(z,w) \\
&\hspace{10pt}+ \frac{\cosh d(x,w)\sinh td(w,z) + \cosh d(x,z)\sinh(1-t)d(w,z) - \cosh d(x,z)\sinh d(z,w)}{\sinh d(z,w)}.
\end{align*}
Putting
\[
L(t) = \cosh d(x,w)\sinh td(w,z) + \cosh d(x,z)\sinh(1-t)d(w,z)
\]
and dividing both sides by $t > 0$, we get
\[
0 \leq f(z,w)+ \frac{1}{t\sinh d(z,w)}(L(t) - \cosh d(x,z)\sinh d(z,w))
\]
Letting $t \searrow 0$, we get
\begin{align*}
0
&\leq f(z,w) + \frac{1}{\sinh d(z,w)} \lim_{t \searrow 0}\frac{L(t)-\cosh d(x,z)\sinh d(z,w)}{t}\\
&= f(z,w) + \frac{1}{\sinh d(z,w)}\lim_{t \searrow 0}\frac{d}{dt}(L(t)-\cosh d(x,z)\sinh d(z,w))\\
&= f(z,w)\\
&\hspace{10pt}+\frac{d(z,w)}{\sinh d(z,w)}\lim_{t \searrow 0}\left(\cosh d(x,w)\cosh td(z,w) - \cosh d(x,z)\cosh(1-t)d(z,w)\right)\\
&= f(z,w) + \frac{d(z,w)}{\sinh d(z,w)}(\cosh d(x,w) - \cosh d(x,z)\cosh d(z,w)).
\end{align*}
Fix $x_1, x_2 \in X$. Let $z_1 \in L_fx_1$ and $z_2 \in L_fx_2$. We suppose $z_1 \neq z_2$. Then, we get
\[
0 \leq f(z_1,z_2) + \frac{d(z_1,z_2)}{\sinh d(z_1,z_2)}(\cosh d(x_1,z_2) - \cosh d(x_1,z_1)\cosh d(z_1,z_2)).
\]
Similarly, it holds that 
\[
0 \leq f(z_2,z_1) + \frac{d(z_1,z_2)}{\sinh d(z_1,z_2)}(\cosh d(x_2,z_1) - \cosh d(x_2,z_2)\cosh d(z_1,z_2)).
\]
Adding these inequalities, we have 
\begin{align*}
0 
&\leq f(z_1,z_2) + f(z_2,z_1)+ \frac{d(z_1,z_2)}{\sinh d(z_1,z_2)}\cosh d(x_1,z_2) - \cosh d(x_1,z_1)\cosh d(z_1,z_2)\\
&\hspace{10pt} + \frac{d(z_1,z_2)}{\sinh d(z_1,z_2)}(\cosh d(x_2,z_1) - \cosh d(x_2,z_2)\cosh d(z_1,z_2))\\
&\leq  \frac{d(z_1,z_2)}{\sinh d(z_1,z_2)}(\cosh d(x_1,z_2) - \cosh d(x_1,z_1)\cosh d(z_1,z_2) )\\
&\hspace{10pt}+ \frac{d(z_1,z_2)}{\sinh d(z_1,z_2)}(\cosh d(x_2,z_1) - \cosh d(x_2,z_2)\cosh d(z_1,z_2)).
\end{align*}
Since $t/(\sinh t)>0$ for $t >0$, we get
\[
(\cosh d(x_1,z_1) + \cosh d(x_2,z_2))\cosh d(z_1,z_2) \leq \cosh d(x_1,z_2) + \cosh d(x_2,z_1).
\]
Notice that this inequality obviously holds when $z_1=z_2$. Using this inequality, we show that $L_f$ is singleton. If $x = x_1 = x_2$, we get
\[
(\cosh d(x,z_1) + \cosh d(x,z_2))\cosh d(z_1,z_2) \leq \cosh d(x,z_2) + \cosh d(x,z_1).
\]
and hence $\cosh d(z_1,z_2) \leq 1$. This implies $z_1 = z_2$. Therefore $L_f$ is singleton. Further, $L_f$ is firmly hyperbolically nonspreading. 

(iii) Let $z \in \Equil f$. Then we get
\begin{align*}
&\inf_{y \in K}(f(z,y) + \cosh d(z,y) -\cosh d(z,z))\\
&= \inf_{y\in K}(f(z,y) + \cosh d(z,y) -1)\\
&\geq \inf_{y\in K}f(z,y) \geq 0
\end{align*}
and hence $z \in F(L_f)$. On the other hand, let $z \in F(L_f)$ and $w \in K$ with $L_fz \neq w$. Then we get
\begin{align*}
0
&\leq f(L_fz,w) + \frac{d(L_fz,w)}{\sinh d(L_fz,w)} (\cosh d(L_fz,w) - \cosh d(L_fz,z)\cosh d(z,w))\\
&\leq f(z,w) + \frac{d(z,w)}{\sinh d(z,w)}(\cosh d(z,w) - \cosh d(z,z)\cosh d(z,w))\\
&= f(z,w).
\end{align*}
This implies $z \in \Equil f$. Therefore we get $\Equil f = F(L_f)$. Further, since firmly hyperbolically nonspreading $L_f$ is quasinonexpansive, we obtain that $\Equil f$ is  closed and  convex.
\end{proof}

Let $f \colon K \times K \to \mathbb{R}$ satisfy Condition \ref{bifunction}. Then the resolvent of an equilibrium problem $L_f$ is defined by
\[
L_fx = \left\{z \in K \varmid \inf_{y \in K}(f(z,y) + \cosh d(x,y) - \cosh d(x,z)) \geq 0\right\}
\]
for all $x \in X$. By Theorem \ref{WelldefinedResolvent}, a mapping $L_f$ is single-valued.

\begin{Ack}
This work was partially supported by JSPS KAKENHI Grant Number JP21K03316.
\end{Ack}

\end{document}